\documentclass[11pt]{article}

\usepackage{amssymb,amsmath, amsfonts, amsthm}

\usepackage{graphicx}
\usepackage{pgfplots}
\usepackage{tikz}
\usepackage{booktabs}
\usepackage{array}
\usepackage{verbatim}
\usepackage{subfig}
\usepackage{geometry}
\usepackage{scrextend}
\usepackage{algorithm}

\usepackage{fancyhdr}

\newcommand{\vertex}{\node[vertex]}
\tikzstyle{vertex}=[circle, draw, inner sep=0pt, minimum size=6pt]

\newtheorem{theorem}{Theorem}
\newtheorem{lemma}{Lemma}
\newtheorem{corollary}{Corollary}
\newtheorem{problem}{Problem}
\newtheorem{proposition}{Proposition}
\newtheorem{conjecture}{Conjecture}

\newtheorem{observation}{Observation}

\begin{document}
\title{The number of independent sets in bipartite graphs and benzenoids}
\author{
Michael Han$^{d}$
\and 
Sycamore Herlihy$^{e}$
\and
Kirsti Kuenzel$^{b}$
\and 
Daniel Martin$^{a}$
\and
Rachel Schmidt$^{c}$}

\date{\today}

\maketitle

\begin{center}
$^a$ Department of Mathematics, Amherst College, Amherst, MA \\
$^b$ Department of Mathematics, Trinity College, Hartford, CT\\
$^c$ Department of Mathematics and Statistics, Williams College, Williamstown, MA\\
$^d$ Department of Mathematics, Yale University, New Haven, CT\\
$^e$ Department of Mathematical Sciences, Worcester Polytechnic Institute, Worcester, MA\\

\end{center}
\vskip15mm
\begin{abstract} 
Given a graph $G$, we study the number of independent sets in $G$, denoted $i(G)$. This parameter is known as both the Merrifield-Simmons index of a graph as well as the Fibonacci number of a graph. In this paper, we give general bounds for $i(G)$ when $G$ is bipartite and we give its exact value when $G$ is a balanced caterpillar. We improve upon a known upper bound for $i(T)$ when $T$ is a tree, and study a conjecture that all but finitely many positive integers represent $i(T)$ for some tree $T$. We also give exact values for $i(G)$ when $G$ is a particular type of benzenoid. 
\end{abstract}

\section{Introduction} \label{sec:intro}
The study of independent sets in graphs is popular throughout the literature. From studying the independence number of a graph (the maximum cardinality of an independent set) to graphs where every maximal independent set is maximum (well-covered graphs) to the chromatic number of a graph (partitioning the vertices into the fewest number of independent sets). Motivated by two separate and completely different applications, a parameter describing  the number of independent sets in a graph was defined. We first see the topological index known as the Merrifield-Simmons index $\sigma(G)$, or the number of independent sets in $G$,  defined in \cite{MS-1989}. The motivation for this index is that there is a correlation between boiling points in certain molecules and the number of independent sets in the graphical representation of the molecule. Independently, the Fibonacci number of a graph $G$ $i(G)$, also the number of independent sets in $G$, is defined in \cite{PT-1982}. In this context, the name refers to the fact that the usual Fibonacci numbers describe the number of independent sets in paths. Throughout this paper, we will use the notation $i(G)$. We also note that  there is a series of papers considering the number of independent sets of a certain size and given minimum degree (see \cite{A-1991, EG-2012, K-2001, Z-2010}).

In this paper, we study the number of independent sets in two different graph classes. First, we study the number of independent sets in bipartite graphs, in large part motivated by a problem which first appeared in \cite{L-1989}. A positive integer $n$ is considered \emph{constructible on trees} if there exists a tree $T$ such that $i(T)=n$. In \cite{PT-1982}, it was shown that given a tree $T$ of order $n$, $F_{n+1} \le i(T) \le 2^{n-1}+1$ where $F_{n+1}$ represents the $(n+1)^{st}$ Fibonacci number. It is known that certain integers are not constructible and the question arose as to whether there are infinitely many integers that are not constructible. Later Wagner \cite{W-2009} showed that almost all trees have an even number of independent sets and posed a conjecture regarding $i(T)$ modulo $m$. Law proved Wagner's conjecture in \cite{Law-2010}, but the original question remains as to whether an infinite number of integers are not constructible. We first study $i(G)$ where $G$ is a bipartite graph and then focus on trees. In particular, we provide an algorithm that provides evidence that there are in fact an infinite number of integers which are not constructible on trees. 

Next, we study $i(G)$ when $G$ is a benzenoid. Recall that a benzenoid graph is a $2$-connected graph in which all inner faces are hexagons (and all hexagons are faces). We study benzenoids which can be embedded in the hexagonal lattice (although there are benzenoid systems which need not be embeddable in the hexagonal lattice).  Some benzenoids are carcinogenic and the ability to break down such molecules is of particular importance in the sustainability of the planet.  Therefore, the molecular applications of the Merrifield-Simmons index are important in this graph class. However, the hexagonal lattice is bipartite and therefore this second graph class is not completely unrelated to the first. We give exact values for $i(G)$ in certain benzenoids. Moreover, our algorithms can be generalized to find bounds on any benzenoid graph. 

The remainder of this paper is organized as follows. In Section~\ref{sec:def}, we give useful definitions and known results used throughout the paper. We study $i(G)$ when $G$ is bipartite in Section~\ref{sec:bipartite}. In Section~\ref{sec:trees}, we improve upon a known upper bound for $i(T)$ when $T$ is a tree and we consider the constructibility question on trees. Finally, we give exact values for $i(G)$ when $G$ is a particular type of benzenoid in Section~\ref{sec:benz}. 

\subsection{Definitions and Preliminaries}\label{sec:def}
Let $G$ be a simple undirected graph with vertex set $V(G)$ and edge set $E(G)$. Throughout this paper, we let $n(G) = |V(G)|$. Given two vertices $u, v \in V(G)$, we let $d(u,v)$ represent the distance between $u$ and $v$ in $G$. The \emph{girth} of $G$ is the length of the shortest cycle and is denoted $g(G)$. A set $M \subseteq E(G)$ is a \emph{matching} if no pair of edges in $M$ share a common vertex. The maximum cardinality among all matchings in $G$ is the \emph{matching number} of $G$ and is denoted $\alpha'(G)$. A set $S\subseteq V(G)$ is said to \textit{independent} if no pair of vertices in $S$ are adjacent in $G$. Typically, one studies the maximum size of an independent set in $G$, i.e. the independence number of $G$, denoted $\alpha(G)$. However, in this paper, we study the number of independent sets in $G$, or $i(G)$. Recall the Fibonacci sequence defined as $F_0=0, F_1=1, F_2=1$ and $F_{n+1} = F_n + F_{n-1}$ for $n \ge 2$. The following result is well-known (and the reason for this parameter to be called the Fibonacci number of a graph.) In the following, $L_n$ represents the $n^{th}$ Lucas number. 

\begin{theorem}\cite{PT-1982}\label{thm:P_nC_n} For all $n \in \mathbb{N}$, $i(P_n) = F_{n+2}$ and $i(C_n) = L_n = F_{n-1} + F_{n+1}$.
\end{theorem}

Bounds on $i(T)$ in terms of the order of $T$ were first given in shown in \cite{PT-1982}.

\begin{theorem}\cite{PT-1982}\label{thm:trees} For any tree $T$ on $n$ vertices, $F_{n+2} \le i(T) \le 2^{n-1}+1$. Moreover, $i(T) = F_{n+2}$ if and only if $T\cong P_n$ and $i(T) = 2^{n-1}+1$ if and only if $T\cong K_{1, n-1}$. 
\end{theorem}

Recall that a graph $G$ is said to be \emph{bipartite} if we can partition $V(G) = A\cup B$ such that $A$ and $B$ are independent. We study bipartite graphs in general, and then restrict our attention to trees (or connected, acyclic graphs), followed by benzenoids that are embeddable in the hexagonal lattice. A graph is \emph{unicyclic} if it contains exactly one cycle. Given a tree $T$, we call a vertex with degree $1$ a \emph{leaf}. A \emph{star}, denoted $K_{1, n}$, is a  vertex with $n$ leaves attached. 

\section{Bipartite Graphs}\label{sec:bipartite}
Recall that the girth of a graph $G$ is the length of the smallest cycle in $G$ and is denoted by $g(G)$. As in \cite{PV-2005}
we shall let $H_{n, k}$ denote the unicyclic graph obtained by attaching $n-k$ leaves to one vertex on a cycle of length $k$. The following was proven by Pederson and Vestergaard \cite{PV-2005}.

\begin{theorem}\cite{PV-2005} Given integers $k\ge 3$ and $n\ge k$, $i(H_{n,k}) = F_{k-1} + 2^{n-k}F_{k+1}$.
\end{theorem}

\begin{theorem}\cite{PV-2005}\label{thm:unicyclic} Let $G$ be a unicyclic graph of order $n$ and cycle length $k$. Then $i(G) \le 2^{n-k}F_{k+1} + F_{k-1}$ and equality holds if and only if $G \cong H_{n, k}$. 
\end{theorem}

We use the above to provide a general upper bound for the number of independent sets in a bipartite graph.

\begin{theorem} If $G$ is a connected, bipartite graph on $n$ vertices with girth $g\geq 4$, then $i(G) \le 2^{n-g}F_{g+1} + F_{g-1}$.
\end{theorem}

\begin{proof} 
 We claim for any $n\ge 4$ and $g\ge 4$ that the bipartite graph with the maximum number of independent sets is $H_{n, g}$. Indeed, let $G$ be bipartite with the maximum number of independent sets among all bipartite graphs with fixed order $n$ and girth $g\ge 4$. Moreover, let $C$ be a cycle of length $g$ in $G$. If $G$ contains another cycle $C'$, then for any edge $e$ on $C'$ and not on $C$, $G-e$ is bipartite with order $n$ and girth $g$ and yet $G-e$ contains the same number of independent sets by our choice of $G$. Therefore, we may assume that $G$ is unicyclic. By Theorem~\ref{thm:unicyclic}, $i(G) \le 2^{n-g}F_{g+1} + F_{g-1}$ and equality holds if and only if $G\cong H_{n, g}$. It now follows that for any bipartite graph $G$ with order $n$ and girth $g\ge 4$ that $i(G) \le i(H_{n, g}) =2^{n-g}F_{g+1}+ F_{g-1}$.
\end{proof}

Next, we note that the following was proven in \cite{PV-2005}.

\begin{theorem}\cite{PV-2005} Let $G$ denote a connected graph. If $G$ is not a tree, then $i(G) \le 3\cdot 2^{n-3} +1$ where equality holds if and only if $G$ is a $4$-cycle or $G \cong H_{n,3}$.
\end{theorem}

We next provide an upper bound when $G$ is a tree and contains a perfect matching. 

\begin{theorem}
If $G$ is a tree of order $2n$ and contains a perfect matching, then $i(G) \le 2\cdot 3^{n-1} + 2^{n-1}$ with equality if and only if $G$ is obtained from a star $K_{1, n}$ by appending one leaf to exactly $n-1$ leaves of $K_{1, n}$. 
\end{theorem}

\begin{proof} Let $M$ be a perfect matching in $G$. Write $V(G) = A \cup B$ and enumerate $A =\{a_1, \dots, a_n\}$ and $B = \{b_1, \dots, b_n\}$ such that $a_ib_i \in M$ for $i \in [n]$. We proceed by induction on $n$. Note that if $n=1$, then $G = K_2$ and $i(K_2) = 3 = 2\cdot 3^0 + 2^0$. Next, suppose $n=2$. It follows that $G\in \{P_4, C_4\}$ and $i(C_4) \le i(P_4)=F_6=8=2\cdot 3^1+ 2^1$. Thus, we shall assume for all trees of order $2(n-1)$ for some $n \in \mathbb{N}$ that contain a perfect matching that $i(G) \le 2\cdot 3^{n-2} + 2^{n-2}$. Now suppose $G$ is a tree of order $2n$ that contains a perfect matching. Let $G'$ be the graph with $V(G') = \{c_1, \dots, c_n\}$ such that $c_ic_j \in E(G')$ if either $a_i$ is adjacent to $b_j$ or $b_i$ is adjacent to $a_j$. Since $G$ is connected, so too is $G'$ and therefore has at least $n-1$ edges. Moreover, since removing edges always increases the number of independent sets, $G$ has the greatest number of independent sets when it has exactly $n-1$ edges outside the perfect matching. This implies that $G'$ is a tree. Reindexing the vertices in $G$ and $G'$ if necessary, we may assume $c_n$ is a leaf in $G'$. Hence, we may assume, after possible reindexing, $b_n$ is a leaf in $G$ and $N_G(a_n) = \{b_n, b_{n-1}\}$. Let $I$ be an independent set of $G$ and consider the following three cases: 1) $I \cap \{a_n, b_n\} = \emptyset$, 2) $I\cap \{a_n, b_n\} = \{b_n\}$, or 3) $I \cap \{a_n, b_n\} = \{a_n\}$. The first two cases reduce to finding the number of independent sets in $G - \{a_n, b_n\}$, which is a tree of order $2(n-1)$ containing a perfect matching. The third case reduces to finding the number of independent sets in $G - \{a_n, b_n, b_{n-1}\}$. This gives 
\begin{eqnarray*}
i(G) &=& 2i(G - \{a_n, b_n\}) + i(G - \{a_n, b_n, b_{n-1}\})\\
&=&2i(G - \{a_n, b_n\}) + i(G-\{a_n, b_n\}) - i\left(G- \left(\{a_n, b_n\}\cup N[b_{n-1}]\right)\right)\\
&=& 3i(G - \{a_n, b_n\}) - i\left(G- \left(\{a_n, b_n\}\cup N[b_{n-1}]\right)\right)\\
&\le&3(2\cdot 3^{n-2}+2^{n-2}) - i\left(G- \left(\{a_n, b_n\}\cup N[b_{n-1}]\right)\right)\\
\end{eqnarray*}
with the induction hypothesis used in the last inequality. To maximize the last line above, we must minimize $i\left(G- \left(\{a_n, b_n\}\cup N[b_{n-1}]\right)\right)$. Note that \[2^{n-2} \le i\left(G- \left(\{a_n, b_n\}\cup N[b_{n-1}]\right)\right)\] as any independent set $I\subseteq \cup_{i=1}^{n-2}\{b_i\}$ is an independent set of $i\left(G- \left(\{a_n, b_n\}\cup N[b_{n-1}]\right)\right)$. Hence, 
\begin{eqnarray*}
i(G) &\le& 3(2\cdot 3^{n-2}+2^{n-2}) - i\left(G- \left(\{a_n, b_n\}\cup N[b_{n-1}]\right)\right)\\
& \le& 3(2\cdot 3^{n-2}+2^{n-2}) -2^{n-2}\\
&=&2\cdot3^{n-1} + 2^{n-1}.
\end{eqnarray*}
Furthermore, if $i(G) = 2\cdot3^{n-1} + 2^{n-1}$, then we have equality throughout the above. This implies $i (G- \{a_n, b_n\}) = 2\cdot 3^{n-2} + 2^{n-2}$ and by the induction hypothesis, $G - \{a_n, b_n\}$ is obtained from the star $K_{1, n-1}$ by appending one leaf to exactly $n-2$ leaves of $K_{1, n-1}$. We also know that  $i\left(G- \left(\{a_n, b_n\}\cup N[b_{n-1}]\right)\right)= 2^{n-2}$ meaning that $b_{n-1}$ is adjacent to each vertex in $A$. That is, $b_{n-1}$ is the center of the star $K_{1, n-1}$ and each vertex of $\{b_1, \dots, b_{n-2}\}$ are leaves in $G - \{a_n, b_n\}$. Moreover, $c_n$ being a leaf in $G'$ implies that the only vertex in $G- \{a_n, b_n\}$ adjacent to $a_n$ is $b_{n-1}$. It follows that $G$ is obtained from the star $K_{1, n}$ by appending one leaf to exactly $n-1$ leaves of $K_{1, n}$. 

\end{proof}

\section{Trees}\label{sec:trees}
First, we begin with an exact value for the number of independent sets in ``balanced" caterpillars. Recall that a caterpillar is a tree $T$ such that if all leaves were removed from $T$, the resulting graph is a path. The underlying path in a caterpillar is commonly referred to as the spine. We will say that a caterpillar is ``balanced" if each vertex of the spine has exactly $k$ leaves for a fixed number $k\in \mathbb{N}$. We use the following result due to Hopkins and Staton \cite{HS-1984}.

\begin{lemma}\cite{HS-1984}\label{lem:paths} The number of independent sets of size $k$ in $P_n$ is 
\[\binom{n-k+1}{k}.\]
\end{lemma}

In the next two results, given a caterpillar $T$, we give the exact value for $i(T)$ if $T$ is  balanced  and we bound $i(T)$ otherwise. 

\begin{proposition}\label{prop:caterpillar} Let $G$ be a caterpillar obtained from the path $P_n= x_1x_2\cdots x_n$ for any $n\ge 2$ by appending exactly $k\ge 1$ leaves to each vertex of $P_n$. Then 
\[i(G) = \sum_{j=0}^{\left\lceil\frac{n}{2}\right\rceil} {{n+1-j}\choose {j}} 2^{(n-j)k}.\]
\end{proposition}

\begin{proof} Let $I$ be any independent set in $G$. Note that $I$ contains $j$ vertices from the spine. From Lemma~\ref{lem:paths}, the number of ways to select $j$ vertices from the spine is ${{n+1-j}\choose {j}}$. The remaining vertices in $I$ are leaves in $G$ that are not adjacent to the $j$ vertices of $I$ on the spine. Since each vertex on the spine has $k$ leaves, there are $k(n-j)$ vertices to choose from. That is, there are $2^{(n-j)k}$ ways to select the remaining vertices in $I$. Summing over all $j$ yields 
\[  \sum_{j=0}^{ \left\lceil\frac{n}{2}\right\rceil} {{n+1-j}\choose {j}} 2^{(n-j)k}.\]

\end{proof}

Note that the following upper bound was given in \cite{PV-2004}. 

\begin{theorem}\cite{PV-2004}\label{thm:upper} Let $T$ denote a tree of order $n\ge 2$ and diameter $d$. Then 
\[i(T) \le F_d + 2^{n-d}F_{d+1}.\]
\end{theorem}

Using the above result, we bound $i(T)$ where $T$ is any caterpillar.

\begin{corollary} Let $G$ be a caterpillar obtained from the path $P_n = x_1x_2\cdots x_n$ for any $n\ge 2$ by appending exactly $k_i \ge 0$ leaves to $x_i$ for $i \in [n]$. If $\ell = \min_{i \in [n]} k_i$, then 

\[\sum_{j=0}^{\left\lceil\frac{n}{2}\right\rceil} {{n+1-j}\choose {j}} 2^{(n-j)\ell} \le i(T) \le F_{n+1} + 2^{|V(T)|-(n+1)}F_{n+2}.\]
\end{corollary}

Before moving on, we would like to point out an application of Proposition~\ref{prop:caterpillar}. Recall from combinatorics that given an alphabet $S$, a word of length $n$ is any string of characters taken from the alphabet $S$. In what follows, we consider $11$-avoiding words taken from the alphabet $\{0, \dots, 2^k\}$ whereby we mean if $1$ is in the $j^{th}$ position of the word, then $1$ is not in the $(j-1)^{st}$ or $(j+1)^{st}$ position of the word.

\begin{proposition}
	Let G be the caterpillar obtained from the path $P_n= x_1x_2\cdots x_n$ by appending exactly $k\ge 0$ leaves to each vertex of $P_n$. Then every independent set $I$ in $G$ corresponds uniquely to a $11$-avoiding word of length $n$ on the alphabet $\{0,\dots, 2^k\}$.
\end{proposition}
\begin{proof}
	We construct a bijection between the objects. Suppose we have $G$ as above, and let $I$ be an independent set in $G$. Denote the leaves of each vertex $x_i$ on the spine of $G$ by $L_i=\{\ell_1,\dots,\ell_k\}$ and assign a label from our alphabet $\{0,\dots, 2^k\}$ to each $x_i$ according to the following rules. Let $\mathcal{P}([k])$ be the power set of the set $\{1, \dots, k\}$ and enumerate the sets of $\mathcal{P}([k])- \emptyset$ as $S_1, \dots, S_{2^k-1}$. Assign $x_i$ the label $0$ if $L_i \cap I = \emptyset$. Assign $x_i$ the label $1$ is $x_i \in I$. Finally, assign $x_i$ the label $j$ if $x_i \not\in I$ and the labels on the vertices of $L_i\cap I$ is exactly the same as the elements in $S_j \in \mathcal{P}([k])- \emptyset$. 
	Since $I$ is an independent set, it is clear that there will be no two adjacent vertices with the label $1$. Furthermore, each $I$ has a unique labeling, as the $2^k$ letters in our alphabet correspond with the $2^k$ ways to choose distinct subsets of $k$ leaves. So the vertex labels of $G$ represent a $11$-avoiding word of length $n$ on the alphabet $\{0,...,2^k\}$. One can easily verify that this mapping is also onto. 
	\end{proof}
	\begin{corollary}
		The number of \emph{11}-avoiding words of length $n$ on the alphabet $\{0,\dots,2^k\}$ is given by \[  \sum_{j=0}^{ \left\lceil\frac{n}{2}\right\rceil} {{n+1-j}\choose {j}} 2^{(n-j)k}.\]
	\end{corollary}

Lastly, we improve the upper bound on $i(T)$ for any nontrivial  tree $T$ given in Theorem~\ref{thm:upper}.

\begin{theorem} Let $T$ be any nontrivial tree of order $n$ that is not a caterpillar and let $P = x_1\cdots x_{d+1}$ be a path of length $d = \rm{diam}(T)$. If $q$ is the matching number of $T - \{x_2, \dots, x_d\}$ and $q \ge 3$, then 
\[i(T) \le 3^q2^{n-d-2q+1}F_{d+1}.\]
\end{theorem}

\begin{proof} Let $M$ be a maximum matching in $T'=T-\{x_2, \dots, x_d\}$. Note that any independent set of $T$ can be written $I = I_1 \cup I_2$ where $I_1\subset \{x_2, \dots, x_d\}$ and $I_2 \subset V(T')$. For each edge $e = uv\in M$, either $u \in I_2$ and $v \not\in I_2$, $u\not\in I_2$ and $v \in I_2$, or $u\not\in I_2$ and $v \not\in I_2$. Moreover, for each $v\in T'$ not incident to an edge in $M$, either $v \in I_2$ or $v\not\in I_2$. Putting this altogether, the maximum number of possibilities for $I_2$ is  $3^q2^{n-(d-1)-2q}$ from which it follows that $i(T) \le 3^q2^{n-(d-1)-2q}F_{d+1}$. To see that this value is better than the bound provided in Theorem~\ref{thm:upper}, note that $3^q<2^{2q-1}$ when $q\ge 3$ meaning that $3^q2^{n-d-2q+1}F_{d+1} <2^{2q-1}2^{n-d-2q+1}F_{d+1} = 2^{n-d}F_{d+1}$.
\end{proof} 

\subsection{Constructibility}
As noted earlier, a positive integer $n$ is considered \emph{constructible on trees} if there exists a tree $T$ such that $i(T)=n$. It is known that certain integers are not constructible; for instance, $4, 6, 7, 10$, and $11$ are the first integers which are not constructible on trees. Linek \cite{L-1989} first posed the following problem.

\begin{problem}\cite{L-1989}\label{conj:constructible} Are there infinitely many numbers which are not the number of independent sets in a tree? 
\end{problem}

In \cite{W-2009}, Wagner showed that almost all trees have an even number of independent sets. Then in \cite{Law-2010}, Law proves the following conjecture posed by Wagner in \cite{W-2009}.
\begin{theorem}\cite{Law-2010} For all positive integers $m$ and $k$, there exists a tree $T$ such that 
\[i(T) \equiv k\pmod{m}.\]
\end{theorem}

We believe that the answer to Problem~\ref{conj:constructible} is yes and provide an algorithm that provides evidence for this. First, we note that in \cite{Law-2010}, this problem is posed as a conjecture, attributed to Wagner, who attributes it to Li, Li, and Wang \cite{LLW-2003}. It is posed as follows.

\begin{conjecture}\cite{LLW-2003} There are   only finitely many positive integers that are not constructible.
\end{conjecture}

We believe the above conjecture is false. For a tree $T$ rooted at $r$, we associate an ordered pair $(a,b)$ where $a$ denotes the number of independent sets containing $r$ and $b$ denotes the number of independent sets not containing $r$. It is clear that $i(T) = a+b$. We now present an algorithmic approach to generate every number that is constructible on trees. First, we introduce the work of Czabarka, Székely, and Wagner \cite{CSW-2018}, which shows that any rooted tree can be generated recursively using two  operations: rooting and merging. Performing the root operation means that a new vertex is added to the tree and connected to the previous root of the tree such that the new vertex becomes the root of the new tree. If a rooted tree's independence index is denoted as the ordered pair $(a,b)$, after rooting, the new tree's independence index is $(b, a+b)$. Merging two trees means that the roots of the trees are overlaid so that the two original trees descend from the same root vertex in the merged tree. If the initial trees have independence indexes $(a_1, b_1)$ and $(a_2, b_2)$, then the merged tree has independence index $(a_1a_2, b_1b_2)$. 

Using these two operations -- rooting and merging with $P_n$ -- we can build any tree. Thus, by starting with $P_2$ and continuously performing rooting and merging operations with all values of $n$, we can theoretically generate all constructible numbers. Although time and space constraints make this impossible to actually achieve in finite run time, the algorithm to do so is below.

\begin{algorithm}[]
\caption{Generate the Tree-Constructible Numbers}
\emph{Input:} Number of levels of the triangle to generate, number of merging operations with the Fibonacci numbers to compute. \\
\emph{Output:} All constructible numbers formed from $l$ levels of the triangle and $n$ Fibonacci merge operations. \\

\textbf{Node class for ordered pairs.}
\begin{quote}
    Initialize a Node $(a, b)$ with reference to its parent Node. Base Node (1,2) has a Null parent Node. \\
    Rooting method: add new vertex at root. 
    \begin{quote}
    	$(a, b) \rightarrow (b, a+b)$
    \end{quote}
    Merging method: iterate through merges with $n$ Fibonaccis.
    \begin{quote}
    	$(a,b) \rightarrow (F_n(a), F_{n+1}(b))$
    \end{quote}
\end{quote}

\textbf{Triangle construction.}
\begin{quote}
    Generate new levels of the triangle by performing rooting and merging on every Node of the current level. 
    \begin{quote}
    	For Node i = $(a_i, b_i)$ in current Level:
    	\begin{quote}
		Add rooted Node $(b_i, a_i+b_i)$ to the next Level. \\
		Add merged Nodes $(F_n(a_i), F_{n+1}(b_i))$ to the next Level.
	\end{quote}
    \end{quote}
    Construct a list of Node sums (constructible numbers) by summing $a+b$ for all Nodes and remove duplicates    
\end{quote}

\textbf{Constructing a tree with a given number of independent sets}
\begin{quote}
    Identify if the inputted number is constructible and find the Node that constructs it by searching the sum $a+b$ for all Nodes in the triangle. Identify that Node as current Node. 
    \begin{quote}
    	While (current Node's parent Node is not Null):
	\begin{quote}
		print(current Node)\\
		current Node = current Node's parent Node
	\end{quote}
    \end{quote}
\end{quote}
\end{algorithm}
\begin{algorithm}
\textbf{Calculating the fraction of non-constructible numbers in a given partition of natural numbers.}
\begin{quote}
    With an inputted partition width $n$, construct partitions: $[1, n], [n+1, 2n], [2n+1, 3n],\dots$ \\
    For Constructible $c$ in Constructible Numbers List:
    \begin{quote}
    	Sort the Constructible number into the correct partition by checking if it is bigger than $n, 2n, 3n, \dots$ 
    \end{quote}
    Calculate partition fractions by dividing the number of constructibles placed into a given partition by $n$. \\
    Non-Constructible fraction: $1$ - Constructible Fraction.
\end{quote}

\end{algorithm}  

The final method in the algorithm that calculates the fraction of non-constructible numbers in a given partition of natural numbers outputs a sequence of the fractions. If the non-constructible numbers were finite, we would expect to see a decreasing sequence of fractions. However, regardless of chosen partition width, we see a highly non-monotonic sequence. Furthermore, the percentage of non-constructible numbers in a given sequence is never less than 30$\%$ regardless of the partition width.

It should be noted that it was shown in \cite{L-1989} that all positive integers are constructible on bipartite graphs. Next, we focus on which positive integers are constructible on forests. That is, for any $n \in \mathbb{N}$, is there a forest $F$ such that $i(F) = n$? Note that although $4$ is not constructible on trees, $i(2K_2) = 2i(K_2) = 4$. Therefore, we also wonder if there are an infinite number of integers that are not constructible on forests. In order to prove that the set of integers which is not constructible on forests is infinite, we would like to better understand when an integer is constructible on either a tree or a forest. To do so, we, as above, look at iteratively constructing trees by appending stars. 

Note first that any tree $T\ne K_{1,n}$ can be obtained from a tree $T'$ and a star $K_{1, n}$ with center vertex $s$ by adding an edge between $s$ and some vertex $x$ in $T'$. Indeed, let $T$ be any tree other than a star and root $T$ at some leaf $\ell$. Define $A_0= \{\ell\}$, $A_i = \{v \in V(T): d(v, \ell)=i\}$ and $t = \max_{v\in T-\ell}d(v, \ell)$. Necessarily, $t\ge 3$ for $T$ is not a star. For any $p \in A_t$, $p$ is necessarily a leaf and its support $s$ is in $A_{t-1}$. It follows that  $s$ has exactly one non-leaf neighbor in $A_{t-2}$, call it $x$ and $T$ is obtained from $T' = T-(N[s]-x)$ by appending the star $K_{1, r}$ with center $s$ where $r = \deg(s)-1$ to $T'$ by adding an edge between $s$ and $x$. 

As in \cite{Law-2010}, given a tree $T$ rooted at $r$, we let $i_0(T, r)$ denote the number of independent sets not including $r$. Further, if we let $r_1, r_2, \dots, r_j$ represent the neighbors of $r$ in $T$, we let $(T_k, r_k)$ be the subtree of $T$ rooted at $r_k$. As shown in \cite{LLW-2003, W-2009}, 
\[i(T) = \prod_{k=1}^ji(T_k) + \prod_{k=1}^ji_0(T_k, r_k).\]

\begin{proposition} If $n\in \mathbb{N}$ is constructible on a tree, then either $n = 2^r+1$ for some $r\in \mathbb{N}$ or we can write $n = 2^ra + b$ where $a$ and $b$ are constructible on forests such that $b \le a \le 2b$ and there exists some tree $T$ with vertex $x$ such that $a=i(T)$ and $b=i(T-x)$.
\end{proposition}

\begin{proof} Suppose $n\in \mathbb{N}$ is constructible and $T$ is such that $i(T)=n$. If $T$ is a star, then $n=2^r+1$ for some $r \in \mathbb{N}$. Thus, we shall assume $T$ is not a star. As noted above, $T$ is obtained from a tree $T'$ and a star $K_{1, t}$ for some $t\in \mathbb{N}$ with center vertex $s$ by adding an between between $s$ and some $x \in V(T')$. Root $T$ as $x$ and label the neighbors of $x$ as $r_1, \dots, r_j, s$ where $r_i \in T'$ for $i\in [j]$. It follows that 
\begin{eqnarray*}
i(T) &=& i(K_{1, t}, s)\prod_{k=1}^ji(T_k, r_k) + i_0(K_{1, t}, s)\prod_{k=1}^ji_0(T_k, r_k)\\
&=&(2^t+1)i(T'-x) + 2^ti(T'-N[x])\\
&=& 2^t(i(T'-x)+i(T'-N[x]) + i(T'-x)\\
&=&2^ti(T') + i(T'-x).
\end{eqnarray*}
Thus, $n=i(T) = 2^ta+b$ where $a$ and $b$ are constructible on forests and $a = i(T')$ and $b = i(T'-x)$. 
\end{proof}

Note: This shows us the following identity for the Fibonacci numbers. Indeed, we know $i(P_n) = F_{n+2} = F_{n+1} + F_n = 2F_n + F_{n-1} = 2i(P_{n-2}) + i(P_{n-3})$.

\section{Benzenoids}\label{sec:benz}
In this section, we study specific types of benzenoids, which are  $2$-connected graphs in which all inner faces are hexagons (and all hexagons are faces). Recall that in a graph $G$, a set $S\subseteq V(G)$ is a vertex cover if every edge in $G$ is incident to some vertex in $S$. It is known that that there is a natural bijection between the set of all vertex covers in a graph and the set of all independent sets in a graph as for any vertex cover $S$ in $G$, $V(G) - S$ is an independent set. We use this relationship to study $i(G)$ when $G$ is a benzenoid.

\begin{theorem}\label{thm:towerhex} If $G$ is  a linear chain of $n\ge 1$ hexagons as shown in Figure \ref{fig:towerhex}, the number of vertex covers $a_n$ is defined recursively as $a_1=18$, $b_1=5$, 
\[a_n = 8a_{n-1}-6b_{n-1}\hskip15mm \text{for }n\ge 2\]
where 
\[b_n = 2a_{n-1} - b_{n-1}\hskip15mm \text{for } n\ge2,\]
and  $b_n$ represents the number of vertex covers containing  exactly one vertex in $\{w, x\}$.  
\end{theorem}

\begin{figure}[h!]
\begin{center}
\begin{tikzpicture}[auto]

	\vertex (0) at (0,0) [label=left:$w$,scale=.75pt]{};
	\vertex (1) at (0,1) [label=left:$x$,scale=.75pt]{};
	\vertex (2) at (.75, 1.5) [label=above:$y$,scale=.75pt]{};
	\vertex (3) at (1.5, 1) [label=above:$z$,scale=.75pt]{};
	\vertex (4) at (1.5, 0) [label=below:$u$,scale=.75pt]{};
	\vertex (5) at (.75, -.5) [label=below:$v$,scale=.75pt]{};
	\vertex (6) at (2.25, 1.5) [label=above:$$,scale=.75pt]{};
	\vertex (7) at (3, 1) [label=above:$$,scale=.75pt]{};
	\vertex (8) at (3, 0) [label=above:$$,scale=.75pt]{};
	\vertex (9) at (2.25, -.5) [label=above:$$,scale=.75pt]{};
	\vertex (10) at (5, 0) [label=above:$$,scale=.75pt]{};
	\vertex (11) at (5, 1) [label=above:$$,scale=.75pt]{};
	\vertex (12) at (5.75, 1.5) [label=above:$$,scale=.75pt]{};
	\vertex (13) at (6.5, 1) [label=above:$$,scale=.75pt]{};
	\vertex (14) at (6.5, 0) [label=above:$$,scale=.75pt]{};
	\vertex (15) at (5.75, -.5) [label=above:$$,scale=.75pt]{};
	\draw[dashed] (3.05,1) -- (3.75, 1.5);
	\draw[dashed] (3.05, 0) -- (3.75, -.5);
	\draw[dashed] (4.95, 1) -- (4.25, 1.5);
	\draw[dashed] (4.95, 0) -- (4.25, -.5);
	
	\path
		(0) edge (1) 
		(1) edge (2)
		(2) edge (3)
		(3) edge (4)
		(4) edge (5)
		(0) edge (5)
		(3) edge (6)
		(6) edge (7)
		(7) edge (8)
		(8) edge (9)
		(9) edge (4)
		(10) edge (11)
		(11) edge (12)
		(12) edge (13)
		(13) edge (14)
		(14) edge (15)
		(15) edge (10)
		;

\end{tikzpicture}
\end{center}
\caption{Linear chain of $n+1$ hexagons }
\label{fig:towerhex}
\end{figure}
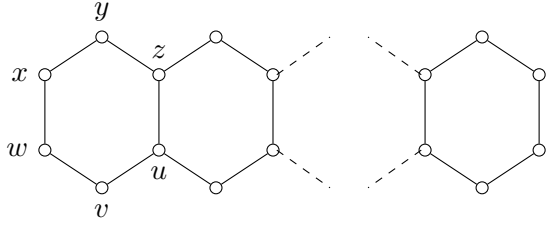

\begin{proof} Note that if $G$ is a single hexagon, then the number of vertex covers $a_1$ is defined to be $18$ as $i(C_6) = F_5+F_7 = 18$ by Theorem~\ref{thm:P_nC_n}. One can easily verify that $b_1=5$. Assume the statement of the theorem is true when $G$ is a linear chain of $n\ge1$ hexagons. Now let $G$ be a linear chain of $n+1$ hexagons  as shown in Figure  \ref{fig:towerhex} and label the vertices in the left most hexagon as shown.  To see why $a_{n+1}= 8a_n - 2b_n$, we need to consider three cases, two of which are equivalent. First, let $J$ be any vertex cover that contains $u$ and not $z$. It follows that $y \in J$ in order to cover the edge $yz$. Further, to cover $vw$ and $wx$, there are $i(P_3)=5$ ways to choose from $\{v, w, x\}$. Since $b_n$ counts the number of vertex covers that contain $u$ and not $z$, the total number of vertex covers in $G$ that contain $u$ and not $z$ is $5b_n$. The case where $J$ contains $z$ and not $u$ is analogous meaning that the number of vertex covers in $G$ containing exactly one of $u$ or $z$ is $10b_n$. Finally, suppose $u$ and $z$ are in $J$. It follows that there are $i(P_4) = 8$ ways to choose vertices from $\{v, w, x, y\}$ to cover the edges $vw, wx$, and $xy$. Therefore, the number of vertex covers in $G$ that contain both $u$ and $z$ is $8(a_n-2b_n)$, giving us a total of $a_{n+1} = 10b_n + 8(a_n-2b_n) = 8a_n - 6b_n$. 

Next, we show that the number of vertex covers in $G$ that contain $w$ and not $x$ is $b_{n+1} = 2a_n - b_n$. As above, we have three cases to consider. First, let $J$ be any vertex cover that contains $w$ and $u$, but not $x$ or $z$. Thus, $y \in J$ in order to cover $xy$ and either $v \in J$ or $v\not\in J$. Thus, there are $2b_n$ possible vertex covers of this type. Next, suppose $J$ contains $w$ and $z$, but not $x$ or $u$. Then both $y$ and $v$ are in $J$ and there are $b_n$ possible vertex covers of this type. Finally, suppose $J$ contains $w, u$, and $z$, but not $x$. Therefore, $y \in J$ to cover $xy$ and $J$ may or may not contain $v$ for a total of $2(a_n-2b_n)$ possible vertex covers of this type. Combining this altogether gives us $b_{n+1} = 3b_n + 2(a_n-2b_n) = 2a_n - b_n$. 

\end{proof}

\begin{corollary}
  Given a linear chain of $n\ge 1$ hexagons, the number of vertex covers is
  \[ a_n = \frac{1}{\sqrt{33}}  \left( \left( 51 + 9 \sqrt{33} \right) 
     \left( \frac{7 + \sqrt{33}}{2} \right)^{n - 1} - \left( 51 - 9 \sqrt{33}
     \right)  \left( \frac{7 - \sqrt{33}}{2} \right)^{n - 1} \right) . \]
     \end{corollary}
  \begin{proof}
    Since $a_{n+1} = 8 a_n - 6 b_n$ and $b_{n+1} = 2 a_n - b_n$,
    we can represent these equations in the following matrix:
    \[ \left[\begin{array}{c}
         a_{n + 1}\\
         b_{n + 1}
       \end{array}\right] = \left[\begin{array}{cc}
         8 & - 6\\
         2 & - 1
       \end{array}\right]  \left[\begin{array}{c}
         a_n\\
         b_n
       \end{array}\right] . \]
    Since $a_1 = 18$ and $b_1 = 5$ we can see by induction that we can
    rewrite the matrix as follows:
    \[ \left[\begin{array}{c}
         a_{n + 1}\\
         b_{n + 1}
       \end{array}\right] = \left[\begin{array}{cc}
         8 & - 6\\
         2 & - 1
       \end{array}\right]^n  \left[\begin{array}{c}
         18\\
         5
       \end{array}\right] . \]
    Diagonlization gives us
    \[ \left[\begin{array}{c}
         a_{n + 1}\\
         b_{n + 1}
       \end{array}\right] = \left( \left[\begin{array}{cc}
         1 & 1\\
         \frac{4}{9 + \sqrt{33}} & \frac{4}{9 - \sqrt{33}}
       \end{array}\right]  \left[\begin{array}{cc}
         \frac{7 + \sqrt{33}}{2} & 0\\
         0 & \frac{7 - \sqrt{33}}{2}
       \end{array}\right]  \left( \frac{1}{\sqrt{33}}  \left[\begin{array}{cc}
         \frac{9 + \sqrt{33}}{2} & - 6\\
         \frac{- 9 + \sqrt{33}}{2} & 6
       \end{array}\right] \right) \right)^n  \left[\begin{array}{c}
         18\\
         5
       \end{array}\right] \]
    \[ = \left[\begin{array}{cc}
         1 & 1\\
         \frac{4}{9 + \sqrt{33}} & \frac{4}{9 - \sqrt{33}}
       \end{array}\right]  \left[\begin{array}{cc}
         \frac{7 + \sqrt{33}}{2} & 0\\
         0 & \frac{7 - \sqrt{33}}{2}
       \end{array}\right]^n  \left( \frac{1}{\sqrt{33}} 
       \left[\begin{array}{cc}
         \frac{9 + \sqrt{33}}{2} & - 6\\
         \frac{- 9 + \sqrt{33}}{2} & 6
       \end{array}\right] \right)  \left[\begin{array}{c}
         18\\
         5
       \end{array}\right] \]
    \[ = \left[\begin{array}{cc}
         1 & 1\\
         \frac{4}{9 + \sqrt{33}} & \frac{4}{9 - \sqrt{33}}
       \end{array}\right]  \left[\begin{array}{cc}
         \left( \frac{7 + \sqrt{33}}{2} \right)^n & 0\\
         0 & \left( \frac{7 - \sqrt{33}}{2} \right)^n
       \end{array}\right]  \left( \frac{1}{\sqrt{33}}  \left[\begin{array}{cc}
         \frac{9 + \sqrt{33}}{2} & - 6\\
         \frac{- 9 + \sqrt{33}}{2} & 6
       \end{array}\right] \right)  \left[\begin{array}{c}
         18\\
         5
       \end{array}\right] . \]
    Computing this gives us the following for $a_{n+1}$:
    \[ a_{n+1} = \frac{1}{\sqrt{33}}  \left( \left( 51 + 9 \sqrt{33} \right)
       \left( \frac{7 + \sqrt{33}}{2} \right)^n - \left( 51 - 9 \sqrt{33}
       \right)  \left( \frac{7 - \sqrt{33}}{2} \right)^n \right) . \]
    This then gives us the desired expression for $a_n$:
    \[ a_n = \frac{1}{\sqrt{33}}  \left( \left( 51 + 9 \sqrt{33} \right) 
       \left( \frac{7 + \sqrt{33}}{2} \right)^{n - 1} - \left( 51 - 9
       \sqrt{33} \right)  \left( \frac{7 - \sqrt{33}}{2} \right)^{n - 1}
       \right) . \]
  \end{proof} 
  \vskip5mm
  \begin{observation}
  	The above corollary gives us the following expression for $b_n$: \[b_n=\frac{1}{\sqrt{33}}(2^{-1-2n}((-3+\sqrt{33})(14-2\sqrt{33})^n+(3+\sqrt{33})(14+2\sqrt{33})^n)).\]
  \end{observation}

  Before we study other benzenoid configurations, we point out a particular phenomenon. Note that the number of $2\times n$ matrices containing a $1$ in the top left entry where all entries are integer values and adjacent entries differ by at most $1$, which we denote as $f_n$, is considered on the Online Encyclopedia of Integer Sequences (OEIS) A138977. $f_n$ satisfies the recurrence relation $f_n = c_n + d_n$ where $c_1=2$, $d_1=1$, $c_{n+1} = 4c_n+d_n$, and $d_{n+1} = 2c_n + 3d_n$. The first values in this sequence are $3, 19, 121, 771, \dots $. Note that the first values in the sequence given by $a_n$ in Theorem~\ref{thm:towerhex} are $18, 114, 726, \dots$. One can easily verify that for $n \le 4$, 
  \[\frac{a_{n+1}}{a_n} = \frac{f_{n+1}}{f_n}.\]
  We believe the above ratio holds for all $n \in \mathbb{N}$ and we state this as a conjecture.
  
  \begin{conjecture} For all $n \in \mathbb{N}$, if $f_n$ counts the number of $2\times n$ matrices containing a $1$ in the top left entry where all entries are integer values and adjacent entries differ by at most $1$ and $a_n$ counts the number of independent sets in a linear chain of $n$ hexagons, then 
    \[\frac{a_{n+1}}{a_n} = \frac{f_{n+1}}{f_n}.\]
    \end{conjecture}
  
  Next, we consider two different types of benzenoids constructed from two different linear chains of hexagons. 
  \begin{theorem} Let $G$ be a linear chain of $n$ hexagons, let $H$ be a linear chain of $m$ hexagons, let $v_1$ and $v_2$ be adjacent vertices of degree $2$ in $G$, and $x_1$ and $x_2$ be adjacent vertices of degree $2$ in $H$.  If  $K$ is the graph constructed from $G, H$, and $K_2=yz$ by adding the edges $v_2y, zx_2, v_1x_1$ (as shown in Figure~\ref{fig:Lhex}) so that the resulting graph is a benzenoid, then 
  \[i(K) = 3a_ma_n-3b_mb_n-a_mb_n-b_ma_n\]
  where $a_n=i(H), a_m=i(G)$, $b_m$ is the number of vertex covers in $H$ containing $v_1$ and not $v_2$, and $b_n$ is the number of vertex covers in $G$ containing $x_1$ and not $x_2$. 
  \end{theorem}
  
  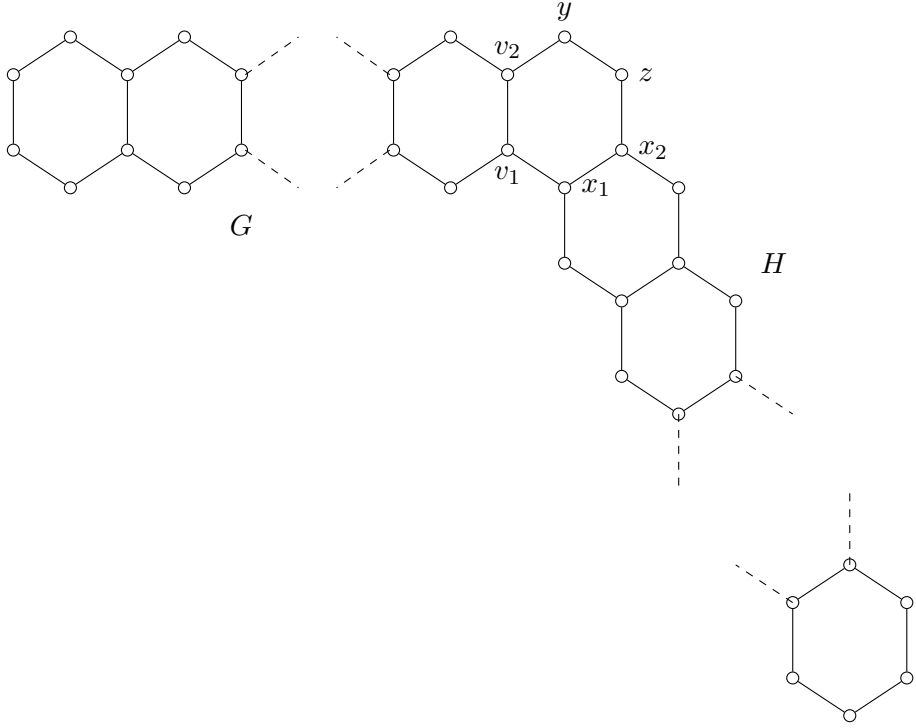
\begin{figure}[h!]
\begin{center}
\begin{tikzpicture}[auto]

	\vertex (0) at (0,0) [label=left:$$,scale=.75pt]{};
	\vertex (1) at (0,1) [label=left:$$,scale=.75pt]{};
	\vertex (2) at (.75, 1.5) [label=above:$$,scale=.75pt]{};
	\vertex (3) at (1.5, 1) [label=above:$$,scale=.75pt]{};
	\vertex (4) at (1.5, 0) [label=below:$$,scale=.75pt]{};
	\vertex (5) at (.75, -.5) [label=below:$$,scale=.75pt]{};
	\vertex (6) at (2.25, 1.5) [label=above:$$,scale=.75pt]{};
	\vertex (7) at (3, 1) [label=above:$$,scale=.75pt]{};
	\vertex (8) at (3, 0) [label=above:$$,scale=.75pt]{};
	\vertex (9) at (2.25, -.5) [label=above:$$,scale=.75pt]{};
	\vertex (10) at (5, 0) [label=above:$$,scale=.75pt]{};
	\vertex (11) at (5, 1) [label=above:$$,scale=.75pt]{};
	\vertex (12) at (5.75, 1.5) [label=above:$$,scale=.75pt]{};
	\vertex (13) at (6.5, 1) [label=above:$v_2$,scale=.75pt]{};
	\vertex (14) at (6.5, 0) [label=below:$v_1$,scale=.75pt]{};
	\vertex (15) at (5.75, -.5) [label=above:$$,scale=.75pt]{};
	\draw[dashed] (3.05,1) -- (3.75, 1.5);
	\draw[dashed] (3.05, 0) -- (3.75, -.5);
	\draw[dashed] (4.95, 1) -- (4.25, 1.5);
	\draw[dashed] (4.95, 0) -- (4.25, -.5);
	\draw[dashed] (9.5, -3)--(10.25, -3.5);
	\draw[dashed] (8.75, -3.5)--(8.75, -4.5);
	\draw[dashed] (11, -5.5) -- (11, -4.5);
	\draw[dashed] (10.25, -6) -- (9.5, -5.5);
	\node(G) at (3, -1)[]{$G$};
	\node(H) at (10, -1.5)[]{$H$};
	\vertex (16) at (7.25, 1.5) [label=above:$y$,scale=.75pt]{};
	\vertex (17) at (8,1) [label=right:$z$,scale=.75pt]{};
	\vertex (18) at (8,0) [label=right:$x_2$,scale=.75pt]{};
	\vertex (19) at (7.25, -.5) [label=right:$x_1$,scale=.75pt]{};
	\vertex (20) at (8.75, -.5) [label=above:$$,scale=.75pt]{};
	\vertex (21) at (8.75, -1.5) [label=above:$$,scale=.75pt]{};
	\vertex (22) at (8, -2) [label=above:$$,scale=.75pt]{};
	\vertex (23) at (7.25, -1.5) [label=above:$$,scale=.75pt]{};
	\vertex (24) at (9.5, -2) [label=above:$$,scale=.75pt]{};
	\vertex (25) at (9.5, -3) [label=above:$$,scale=.75pt]{};
	\vertex (26) at (8.75, -3.5) [label=above:$$,scale=.75pt]{};
	\vertex (27) at (8, -3) [label=above:$$,scale=.75pt]{};
	\vertex (28) at (11, -5.5) [label=above:$$,scale=.75pt]{};
	\vertex (29) at (11.75, -6) [label=above:$$,scale=.75pt]{};
	\vertex (30) at (11.75, -7) [label=above:$$,scale=.75pt]{};
	\vertex (31) at (11, -7.5) [label=above:$$,scale=.75pt]{};
	\vertex (32) at (10.25, -7) [label=above:$$,scale=.75pt]{};
	\vertex (33) at (10.25, -6) [label=above:$$,scale=.75pt]{};
	
	\path
		(0) edge (1) 
		(1) edge (2)
		(2) edge (3)
		(3) edge (4)
		(4) edge (5)
		(0) edge (5)
		(3) edge (6)
		(6) edge (7)
		(7) edge (8)
		(8) edge (9)
		(9) edge (4)
		(10) edge (11)
		(11) edge (12)
		(12) edge (13)
		(13) edge (14)
		(14) edge (15)
		(15) edge (10)
		(13) edge (16)
		(16) edge (17)
		(17) edge (18)
		(18) edge (19)
		(19) edge (14)
		(18) edge (20)
		(20) edge (21)
		(21) edge (22)
		(22) edge (23)
		(23) edge (19)
		(21) edge (24)
		(24) edge (25)
		(25) edge (26)
		(26) edge (27)
		(27) edge (22)
		(28) edge (29)
		(29) edge (30)
		(30) edge (31)
		(31) edge (32)
		(32) edge (33)
		(33) edge (28)

		;

\end{tikzpicture}
\end{center}
\caption{The graph $K$ constructed from two linear chains of hexagons }
\label{fig:Lhex}
\end{figure}
  
  \begin{proof} As in the proof of Theorem~\ref{thm:towerhex}, we count the number of vertex covers in $K$. Since we know that at least one of $\{v_1, v_2\}$ and at least one of $\{x_1, x_2\}$ must be contained in any vertex cover of $K$, we have eight cases to consider. Let $I$ be a vertex cover of $K$. Suppose first that $v_1$ and $x_1$ are in $I$, but neither $v_2$ nor $x_2$ are in $I$. In this case, $y$ and $z$ must both be in $I$ to cover the edges $\{yv_2, yz, zx_2\}$. Since $I\cap V(G)$ is a vertex cover of $G$ containing $v_1$ and not $v_2$, and $I \cap V(H)$ is a vertex cover of $H$ containing $x_1$ and not $x_2$, it follows that there are $b_mb_n$ possibilities for $I$. Next, let us consider the case where $v_1$ and $x_2$ are in $I$, but neither $v_2$ not $x_1$ are in $I$. It follows that $y \in I$ and $z$ may or may not be in $I$. Thus, there are $2b_nb_m$ possibilities for $I$. Analogously, if $v_2$ and $x_1$ are in $I$, and neither $v_1$ nor $x_2$ are in $I$, we have $2b_nb_m$ possibilities for $I$. Moreover, it is impossible for $I \cap \{v_1, v_2, x_1, x_2\} = \{v_2, x_2\}$. Hence, if $|I \cap \{v_1, v_2, x_1, x_2\}| = 2$, then there are $5b_nb_m$ possibilities. 
  
  Let us now consider the five cases where $|I \cap \{v_1, v_2, x_1, x_2\}|\in \{3, 4\}$. If $I \cap\{v_1, v_2, x_1, x_2\}= \{v_1, v_2, x_1\}$, then $z\in I$ and $I$ may or may not contain $y$. Since $I\cap V(H)$ is a vertex cover of $H$ containing both $v_1$ and $v_2$ (of which there are $a_m-2b_m$ possible vertex covers) and $I \cap V(G)$ is a vertex cover of $G$ containing $x_1$ and not $x_2$ (of which there are $b_n$ possible vertex covers), there are a total of $2b_n(a_m-2b_m)$ possibilities for $I$. Similarly, if $I \cap\{v_1, v_2, x_1, x_2\}=\{v_1, x_1, x_2\}$, then $y \in I$ and there are $2b_m(a_n-2b_n)$ possibilities for $I$. If $I \cap\{v_1, v_2, x_1, x_2\}=\{v_1, v_2, x_2\}$, then there are $3b_n(a_m-2b_m)$ possibilities, depending on whether $y$ and $z$ are in $I$. Similarly, if $I \cap\{v_1, v_2, x_1, x_2\}=\{v_2, x_1, x_2\}$, there are $3b_m(a_n-2b_n)$ possibilities for $I$. Finally, if $I \cap\{v_1, v_2, x_1, x_2\}= \{v_1, v_2, x_1, x_2\}$, then there are $3(a_m-2b_m)(a_n-2b_n)$ possibilities. Putting this all together, there are a total of 
  \[3a_ma_n-3b_mb_n-a_mb_n-b_ma_n\]
  different vertex covers of $K$.
  
  \end{proof}
  
  \begin{theorem}\label{thm:Rhex} Let $K$ be the benzenoid obtained from the two linear chains of hexagons, $G$ containing $m+1$ hexagons for $m\ge 1$ and $H$ containing $n+2$ hexagons for $n\ge 1$ both labeled as in Figure~\ref{fig:Rhex}, by identifying three pairs of vertices: $x_5$ and $v_{10}$, $x_4$ and $v_5$, and $x_3$ and $v_6$. Let $G' = G - \{x_3, x_4, x_5, x_6\}$, $H' = H - \{v_3, v_4, v_5, v_6, v_7, v_8, v_9, v_{10}\}$, $a_m=i(G')$, $a_n= i(H')$, $b_m$ represent the number of vertex covers of $G'$ that contain exactly one of $\{x_1, x_2\}$ and $b_n$ represent the number of vertex covers of $H'$ that contain exactly one of $\{v_1, v_2\}$. Then 
  \[i(K) = 88a_ma_n + 75b_mb_n - 70a_mb_n - 70b_ma_n.\]
  \end{theorem}
  
    \begin{figure}[h!]
\begin{center}
\begin{tikzpicture}[auto]

	\vertex (0) at (0,0) [label=left:$$,scale=.75pt]{};
	\vertex (1) at (1, 0) [label=left:$$,scale=.75pt]{};
	\vertex (2) at (-.5, .75) [label=above:$$,scale=.75pt]{};
	\vertex (3) at (0, 1.5) [label=above:$$,scale=.75pt]{};
	\vertex (4) at (1, 1.5) [label=below:$x_1$,scale=.75pt]{};
	\vertex (5) at (1.5, .75) [label=below:$x_2$,scale=.75pt]{};
	\vertex (6) at (1.5, 2.25) [label=above:$x_6$,scale=.75pt]{};
	\vertex (7) at (2.5, 2.25) [label=above:$x_5$,scale=.75pt]{};
	\vertex (8) at (3, 1.5) [label=right:$x_4$,scale=.75pt]{};
	\vertex (9) at (2.5, .75) [label=below:$x_3$,scale=.75pt]{};
	\draw[dashed] (-.5, .75) -- (-1.5, .75);
	\draw[dashed] (0, 0) -- (-.5, -.75);
	\node(G) at (-.5, 2.5)[]{$G$};
	\node(H) at (8.5, 0)[]{$H$};
	\draw[dashed] (7, -2.25) -- (7.5, -3);
	\draw[dashed] (6, -2.25) -- (5.5, -3);
	\vertex (10) at (6, 2.25) [label=above:$v_9$,scale=.75pt]{};
	\vertex (11) at (7, 2.25) [label=above:$v_8$,scale=.75pt]{};
	\vertex (12) at (7.5, 1.5) [label=right:$v_7$,scale=.75pt]{};
	\vertex (13) at (7, .75) [label=right:$v_4$,scale=.75pt]{};
	\vertex (14) at (6, .75) [label=left:$v_5$,scale=.75pt]{};
	\vertex (15) at (5.5, 1.5) [label=left:$v_{10}$,scale=.75pt]{};
	\vertex (16) at (7.5, 0) [label=right:$v_3$,scale=.75pt]{};
	\vertex (17) at (7, -.75) [label=right:$v_2$,scale=.75pt]{};
	\vertex (18) at (7.5, -1.5) [label=below:$$,scale=.75pt]{};
	\vertex (19) at (7, -2.25) [label=below:$$,scale=.75pt]{};
	\vertex (20) at (6, -2.25) [label=below:$$,scale=.75pt]{};
	\vertex (21) at (5.5, -1.5) [label=below:$$,scale=.75pt]{};
	\vertex (22) at (6, -.75) [label=left:$v_1$,scale=.75pt]{};
	\vertex (23) at (5.5, 0) [label=left:$v_6$,scale=.75pt]{};

	\path 
	
	 (0) edge (1)
	 (0) edge (2)
	 (2) edge (3)
	 (3) edge (4)
	 (4) edge (5)
	 (5) edge (1)
	 (4) edge (6)
	 (6) edge (7)
	 (7) edge (8)
	 (8) edge (9)
	 (9) edge (5)
	 (10) edge (11)
	 (11) edge (12)
	 (12) edge (13)
	 (13) edge (14)
	 (14) edge (15)
	 (15) edge (10)
	 (13) edge (16)
	 (16) edge (17)
	 (17) edge (18)
	 (18) edge (19)
	 (19) edge (20)
	 (20) edge (21)
	 (21) edge (22)
	 (22) edge (23)
	 (23) edge (14)
	 (17) edge (22)

	;
 \end{tikzpicture}
\end{center}
\caption{The graph $K$ constructed from two linear chains of hexagons in Theorem~\ref{thm:Rhex}}
\label{fig:Rhex}
\end{figure}
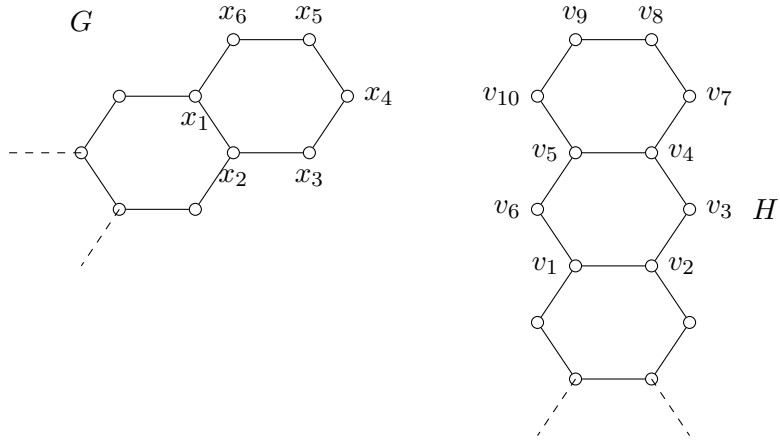 
  \begin{proof} Label the identified vertices in $K$ as shown in Figure~\ref{fig:Rhex}. In order to count the number of vertex covers, let us consider only the three hexagons whose vertices are labeled. Since we know that either $v_1$ or $v_2$ must be in a vertex cover, and likewise $x_1$ or $x_2$ must be in a vertex cover, we have nine cases to consider. Let us first consider the cases in which two of these four vertices are included. First, let us consider the case where $v_1$ and $x_1$ are included in the vertex cover, and $v_2$ and $x_2$ are not included. In this case, we know that $k$ and $v_3$ must be included in the vertex cover. Then the number of ways to pick the rest of the vertices is the number of vertex covers on a single hexagon that contain $d$, which is $13$, times $2$ to account for $x_6$ plus the number that do not contain $d$, which is $5$. This then gives us $31$ different configurations in this case. The cases where $v_1$ and $x_2$ are included (and $v_2$ and $x_1$ not included) and where $v_2$ and $x_2$ are included (and $v_1$ and $x_1$ not included) analogously give us $31$ different configurations. Let us now consider the case where $v_2$ and $x_1$ are included in the vertex cover and $v_1$ and $x_2$ are not included. In this case we know that $k$ must be included in our vertex cover. Notice that of the $18$ different vertex covers of the hexagon induced by $\{j, d, v_4, v_7, v_8, v_9\}$, three of them contain $d$ and do not contain $v_4$, three of them contain $v_4$ and do not contain $d$, ten of them contain both $d$ and $v_4$, and two of them contain neither. With this we can then see that the total number of configurations in this case is $(3 + 3) 2 + 10 (4) + 2 = 54$. Therefore, the total number of vertex covers containing exactly one vertex from $\{v_1, v_2\}$ and exactly one vertex from $\{x_1, x_2\}$ is $(31+31+54+31)b_nb_m$.
  
  Now consider the number of vertex covers containing exactly three vertices from $\{v_1, v_2, x_1, x_2\}$. Notice that the cases where $v_1, v_2$, and $x_1$ are included and where $v_2, x_1$, and $x_2$ are included will also give us $54$ configurations by the same reasoning. Next, let us consider the case where $v_1, v_2$, and $x_2$ are included in the vertex cover. In this case we know that $x_6$ must be included in our vertex cover. Notice that of the $18$ different vertex covers of the top hexagon, five of them contain $j$ and not $v_4$, five of them contain $v_4$ and not $j$, and eight of them contain both $v_4$ and $j$. With this we can see that the total number of configurations in this case is $(5 + 5) 2 + 8 (4) = 52$. Now notice that the case where $v_1, x_1$, and $x_2$ are included will also give us $52$ configurations by the same reasoning. This gives us a total of \[(54+52)b_m(a_n-2b_n)+(52+54)b_n(a_m-2b_m)\] different vertex covers containing exactly three vertices from $\{v_1, v_2, x_1, x_2\}$.
  
  Finally, let us consider the case where $v_1, v_2, x_1$, and $x_2 $ are all included in the vertex cover. Notice that of the $18$ different vertex covers of the top hexagon, two of them contain $j$ and not d and $v_4$, three of them contain $d$ and $j$ and not $v_4$, five of them contain $d$ and $v_4$ and not $j$, three of them contain $j$ and $v_4$ and not $d$, and five of them contain $d, v_4$, and $j$. With this we can see that the total number of configurations in this case is $2(2)+(3+5+3)22+5(23)=88(a_n-2n_n)(a_m-2b_m)$. Having now checked each case, we can conclude that \[i(K) = 88a_ma_n+75b_mb_n-70a_mb_n-70b_ma_n.\]
 \end{proof}

 Finally, we generalize Theorem~\ref{thm:towerhex} to account for a linear chain of  $s$-gons where $s\ge 4$ is an even integer. In this case, we define a linear chain of $s$-gons where $s\ge 4$ recursively as follows. Let $G$ and $H$ be two regular $s$-gons where a pair of parallel edges are drawn vertically. Identify one vertical edge from $G$ with one vertical edge from $H$. Recursively identify vertical edges to create a linear chain of said $s$-gons. 
  
\begin{theorem} Let $G$ be linear chain of $n\ge 1$ $s$-gons where $s\ge 4$ is an even integer. Further, label the vertices on a vertical edge (the vertical edge if $n\ge 2$) of the  left most $s$-gon whose vertices have degree $2$ in $G$ as $w$ and $x$ as in Figure~\ref{fig:towerhex}. The number of vertex covers in $G$ is given by $a_n$ where $a_1= F_{s+1} + F_{s-1}$, $b_1 = F_{s-1}$, \[a_n = F_sa_{n-1} - 2F_{s-2}b_n\hskip15mm\text{ for $n\ge 2$}\] and 
\[b_n = \left(F_{\frac{s-2}{2}}\right)\left(F_{\frac{s}{2}}\right)a_{n-1} + \left(\left(F_{\frac{s-4}{2}}\right)\left(F_{\frac{s}{2}}\right) - F_{\frac{s-2}{2}}\left(F_{\frac{s-4}{2}}+F_{\frac{s}{2}}\right)\right)b_{n-1}\]
and $b_n$ represents the number of vertex covers that contain exactly one vertex in  $\{w, x\}$.
\end{theorem}

\begin{proof} Note that by Theorem~\ref{thm:P_nC_n}, $i(C_n) = F_{n-1}+F_{n+1}$ which implies  $a_1 = F_{s+1} + F_{s-1}$. To see that $b_1= F_{s-1}$, let $wxz$ be a path in $G$. Any vertex cover that contains $w$ and does not contain $x$ must contain $z$, leaving $i(P_{s-3}) = F_{s-1}$ possibilities for the remaining vertices in the vertex cover. Therefore, assume the statement of the theorem is true for a linear chain of $n\ge 1$ $s$-gons and consider when $G$ is a linear chain of $n+1$ $s$-gons. Label the vertices of the $(n+1)^{st}$ $s$-gon as $v_1 \dots v_s$ where $v_{1}$ and $v_{2}$ have degree $3$ in $G$. To see why $a_{n+1}=F_sa_n - 2F_{s-2}b_n$, we consider three cases, two of which are equivalent. First, let us find the number of vertex covers that contain $v_{2}$ and do not contain $v_{1}$. In this case, we know $v_{s}$ must be in the vertex cover. The number of vertex covers of this type is the number of vertex covers of the path $v_3\dots v_{s-1}$, which is $i(P_{s-3}) = F_{s-1}$, times the number of vertex covers of $G - \{v_3, \dots, v_s\}$ that contain $v_2$ and not $v_1$, which is $b_n$. Thus, the total number of vertex covers containing $v_2$ and not $v_1$ is $F_{s-1}b_n$. Similarly, the number of vertex covers that contain $v_1$ and do not contain $v_2$ is $F_{s-1}b_n$.  Now let us consider the case where both $v_1$ and $v_2$ are contained in the vertex cover. The number of vertex covers of this type is the number of vertex covers of the path $v_3\dots v_s$, which is $i(P_{s-2}) = F_s$, times the number of vertex covers of $G - \{v_3, \dots, v_s\}$ that contain both $v_1$ and $v_2$, which is $a_n - 2b_n$. Combining this altogether, we get 
\[a_{n+1} = 2F_{s-1}b_n + F_s(a_n-2b_n) = F_sa_n -2F_{s-2}b_n.\]

Let us now show that \[b_{n+1}=\left(F_{\frac{s-2}{2}}\right)\left(F_{\frac{s}{2}}\right)a_n+\left(\left(F_{\frac{s-4}{2}}\right)\left(F_{\frac{s}{2}}\right) -F_{\frac{s-2}{2}}\left(F_{\frac{s-4}{2}}+F_{\frac{s}{2}}\right)\right)b_n,\]
which is the number of vertex covers of $G$ that contain $v_{\frac{s}{2}+1}$ and not $v_{\frac{s}{2}+2}$. Suppose first that $s=4$. Therefore, $b_{n+1}$ counts the number of vertex covers that contain $v_3$ and not $v_4$. In this case, $v_1$ must be in the vertex cover. Since there are $b_n$ vertex covers of $G - \{v_3, v_4\}$ that contain $v_1$ and not $v_2$ and $a_n - 2b_n$ vertex covers of $G - \{v_3, v_4\}$ that contain $v_1$ and $v_2$, we have 
\[b_{n+1}= a_n - b_n = (F_1)(F_2)a_n + \left((F_0)(F_2) - F_1(F_0 + F_2)\right)b_n.\]

So we may assume that $s\ge 6$. Moreover, by Theorem~\ref{thm:towerhex}, we may assume $s\ge 8$. Let $S$ be such a vertex cover that contains 
 $v_{\frac{s}{2}+1}$ and not $v_{\frac{s}{2}+2}$. We have three cases to consider. Notice that in each case $v_{\frac{s}{2}+3} \in S$.  First, let us assume $v_2\not\in S$ and $v_1 \in S$. It follows that $v_3 \in S$. Therefore, we have the number of vertex covers of the paths $v_4\cdots v_{\frac{s}{2}}$, of which there are $i(P_{\frac{s-6}{2}}) = F_{\frac{s-2}{2}}$, times the number of vertex covers of the path $v_{\frac{s}{2}+4}\cdots v_s$, of which there are $F_{\frac{s-2}{2}}$, times the number of vertex covers of $G-\{v_3, \dots, v_s\}$ that contain $v_1$ and not $v_2$, of which there are $b_n$. 
 Next, assume $v_1\not\in S$ and $v_2 \in S$. It follows that $v_s \in S$. Thus, we have the number of vertex covers of the paths $v_3\cdots v_{\frac{s}{2}}$, of which there are $i(P_{\frac{s-4}{2}}) = F_{\frac{s}{2}}$, times the number of vertex covers of the path $v_{\frac{s}{2}+4}\cdots v_{s-1}$, of which there are $i(P_{\frac{s-8}{2}}) = F_{\frac{s-4}{2}}$, times the number of vertex covers of $G- \{v_3, \dots, v_s\}$ that contain $v_2$ and not $v_1$, of which there are $b_n$. Finally, suppose $v_1\in S$ and $v_2 \in S$. One can easily verify there are $\left(F_{\frac{s-2}{2}}\right)\left(F_{\frac{s}{2}}\right)(a_n - 2b_n)$ possibilities for $S$ in this case. Hence, 
 \begin{eqnarray*}
 b_{n+1} &=& \left(F_{\frac{s-2}{2}}\right)\left(F_{\frac{s-2}{2}}\right)b_n + \left(F_{\frac{s-4}{2}}\right)\left(F_{\frac{s}{2}}\right)b_n + \left(F_{\frac{s-2}{2}}\right)\left(F_{\frac{s}{2}}\right)(a_n - 2b_n)\\
 &=& \left(F_{\frac{s-2}{2}}\right)\left(F_{\frac{s}{2}}\right)a_{n} + \left(\left(F_{\frac{s-4}{2}}\right)\left(F_{\frac{s}{2}}\right) - F_{\frac{s-2}{2}}\left(F_{\frac{s-4}{2}}+F_{\frac{s}{2}}\right)\right)b_{n}.
 \end{eqnarray*}

\end{proof}

\end{document}